\documentclass[a4paper, 9.5pt, final]{amsart}

\usepackage[latin1]{inputenc}
\usepackage[T1]{fontenc}
\usepackage[english]{babel}
\usepackage{amsmath}

\usepackage{lmodern}

\usepackage[extra]{tipa}

\usepackage[OT2,T1]{fontenc}

\usepackage{amsmath, amsthm, amsfonts}

\DeclareMathOperator{\I}{I}

\DeclareMathOperator{\har}{har}

\DeclareMathOperator{\KZ}{KZ}

\DeclareMathOperator{\dR}{dR}

\DeclareMathOperator{\Hom}{Hom}

\DeclareMathOperator{\un}{un}

\DeclareMathOperator{\Li}{Li}

\DeclareMathOperator{\PGL}{PGL}

\theoremstyle{definition}

\newtheorem{Théorème}{Théorème}[section]
\newtheorem{Proposition}[Théorème]{Proposition}

\newtheorem{Notation}[Théorème]{Notation}

\newtheorem{Nota Bene}[Théorème]{Nota Bene}

\newtheorem{Remark}[Théorème]{Remark}

\newtheorem{Example}[Théorème]{Example}

\newtheorem{Definition}[Théorème]{Definition}

\newtheorem{Proposition-Definition}[Théorème]{Proposition-Definition}

\newtheorem{N.B.}[Théorème]{N.B.}

\input cyracc.def
\DeclareFontFamily{U}{russian}{}
\DeclareFontShape{U}{russian}{m}{n}
        { <5><6> wncyr5
        <7><8><9> wncyr7
        <10><10.95><12><14.4><17.28><20.74><24.88> wncyr10 }{}
\DeclareSymbolFont{Russian}{U}{russian}{m}{n}
\DeclareSymbolFontAlphabet{\mathcyr}{Russian}
\makeatletter
\let\@math@cyr\mathcyr
\renewcommand{\mathcyr}[1]{\@math@cyr{\cyracc #1}}
\makeatother
\newcommand{\sh}{\mathcyr{sh}} 

\makeatletter
\newcounter {subsubsubsection}[subsubsection]
\renewcommand\thesubsubsubsection{\thesubsubsection .\@alph\c@subsubsubsection}
\newcommand\subsubsubsection{\@startsection{subsubsubsection}{4}{\z@}%
                                     {-3.25ex\@plus -1ex \@minus -.2ex}%
                                     {1.5ex \@plus .2ex}%
                                     {\normalfont\normalsize\bfseries}}
\newcommand*\l@subsubsubsection{\@dottedtocline{3}{10.0em}{4.1em}}
\newcommand*{\subsubsubsectionmark}[1]{}
\makeatother

\usepackage[top=2.5cm, bottom=2.5cm, left=2.5cm, right=2.5cm]{geometry}

\address{Universit\'{e} de Gen\`{e}ve, Section de math\'{e}matiques, 2-4 rue du Li`{e}vre,
Case postale 64
1211 Gen\`{e}ve 4, Suisse}
\email{david.jarossay@unige.ch}

\begin{document}

\author{David Jarossay}

\title{On generic double shuffle relations, localized multiple polylogarithms and algebraic functions}

\maketitle

\begin{abstract}
We define subvarieties of $\mathcal{M}_{0,n}$ equipped with algebraic functions that are solutions to the generic double shuffle equations satisfied by multiple polylogarithms on $\mathcal{M}_{0,n}$.
\end{abstract}

\noindent

\tableofcontents

\section{Introduction}

\subsection{Multiple zeta values and double shuffle relations}

Multiple zeta values are the following real numbers, for any positive integers $d$ and 
$n_{i}$ ($1 \leq i \leq d$) with $n_{d} \geq 2$ :
\begin{equation}
\label{eq:1} \zeta(n_{1},\ldots,n_{d}) 
= \sum_{0<m_{1}<\ldots<m_{d}} \frac{1}{m_{1}^{n_{1}} \ldots m_{d}^{n_{d}}} 
\end{equation}
\noindent Denoting by $n = \sum_{i=1}^{d} n_{i}$, and $(\epsilon_{n},\ldots,\epsilon_{1}) 
= (\overbrace{0 \ldots 0}^{n_{d}-1}, 1,\ldots,\overbrace{0 \ldots 0}^{n_{1}-1}, 1)$, one has 
\begin{equation}
\label{eq:2} \zeta(n_{1},\ldots,n_{d})  = (-1)^{d} \int_{t_{n}=0}^{1} \frac{dt_{n}}{t_{n}-\epsilon_{n}} \cdots \int_{t_{1}=0}^{t_{2}} \frac{dt_{1}}{t_{1}-\epsilon_{1}} 
\end{equation}

Details on what follows are given in \cite{IKZ} Multiple zeta values satisfy : 

(i) The quasi-shuffle relation : an iterated sum of series as in (\ref{eq:1}) is canonically equal to a $\mathbb{Z}$-linear combination of iterated sums of series 
(for instance $\sum\limits_{0<n_{1}} \times \sum\limits_{0<n'_{1}} = 
\sum\limits_{0<n_{1}=n'_{1}} + \sum\limits_{0<n_{1}<n'_{1}} + \sum\limits_{0<n'_{1}<n_{1}}$). This gives, for any 
$w=(n_{1},\ldots,n_{d})$ and $w'=(n'_{1},\ldots,{n'}_{d'})$, 
\begin{equation} \label{eq:first shuffle}\zeta(w)\zeta(w')=\zeta(w \ast w')
\end{equation}
where $\ast$ is an operation on indices called the quasi-shuffle procuct. 

(ii) The shuffle relation : a product of two integrals on simplices is canonically equal to a $\mathbb{Z}$-linear combination of integrals on simplices : for any positive integers $n,n'$, $\int_{0<t_{1}<\ldots < t_{n}<1} \int_{0<t_{n+1}<\ldots t_{n+n'}<1} =$
$\sum\limits_{\substack{\sigma \text{ permutation of }\{1,\ldots,l+l'\} \\\text{s.t. } \sigma(1)<\ldots<\sigma(n)\\ \text{and } \sigma(l+l)<\ldots<\sigma(n+n')}}
\int_{0<t_{\sigma^{-1}(1)}<\ldots < t_{\sigma^{-1}(n+n')<1}}$. This gives, for any 
$w=(n_{1},\ldots,n_{d})$ and $w'=(n'_{1},\ldots,{n'}_{d'})$, 
\begin{equation} \label{eq:second shuffle}\zeta(w)\zeta(w')=\zeta(w\text{ }\sh\text{ }w') 
\end{equation}
where $\sh$ is an operation on indices called the shuffle product. 

The double shuffle relations are equations (\ref{eq:first shuffle}) and (\ref{eq:second shuffle}). When one adds to them some data about the regularization of iterated integrals and iterated series in the $n_{d}=1$ case, they conjecturally generate all polynomial equations over $\mathbb{Q}$ satisfied by multiple zeta values.

\subsection{Multiple polylogarithms and generic double shuffle relations}

We are going to work not with multiple zeta values but with multiple polylogarithms, introduced in \cite{Goncharov}, because it makes the setting more flexible. They can be defined as power series expansions : for $z_{1},\ldots,z_{d}$ in $\mathbb{C} - \{0\}$, and for $z$ in $\mathbb{C}$ close to $0$,

\begin{equation} \label{eq:1Li} \Li((n_{i})_{d};(z_{i})_{d})(z) = \sum_{0<m_{1}<\ldots<m_{d}} 
\frac{(\frac{z_{2}}{z_{1}})^{m_{1}} \ldots (\frac{z}{z_{d}})^{m_{d}}}{m_{1}^{n_{1}} \ldots m_{d}^{n_{d}}} 
\end{equation}

They can also more generally be defined as iterated integrals which gives an analytical continuation of multiple polylogarithms (\cite{Goncharov}, Theorem 2.2) :
for $n = \sum_{i=1}^{d} n_{i}$, and $(\epsilon_{n},\ldots,\epsilon_{1})
= (\overbrace{0 \ldots 0}^{n_{d}-1},z_{d} ,\ldots,\overbrace{0 \ldots 0}^{n_{1}-1},z_{1})$

\begin{equation}
\label{eq:2Li} \Li((n_{i})_{d};(z_{i})_{d})(\gamma) = (-1)^{d} \int_{t_{n}=0}^{1} \frac{dt_{n}}{t_{n}-\epsilon_{n}}) \cdots \int_{t_{1}=0}^{t_{2}} \frac{dt_{1}}{t_{1}-\epsilon_{1}})
\end{equation}

By these two expressions, they satisfy an extension of equations (\ref{eq:first shuffle}) and (\ref{eq:second shuffle}), proved similarly, which we call \emph{generic double shuffle equations}. Thus we will also write, for all indices $w,w'$,

\begin{equation}
\label{eq:generic1}  \Li(w)\Li(w') = \Li(w \ast w') \end{equation}
\begin{equation}
\label{eq:generic2}  \Li(w)\Li(w') = \Li(w\text{ }\sh\text{ }w') \end{equation}

The precise formulation is reviewed in \S2.

\subsection{Motivation from the study of $p$-adic multiple zeta values}

The motivation for this paper comes from the study of $p$-adic multiple zeta values (and their cyclotomic generalizations), as in \cite{J1}, \cite{J2}, \cite{J3}, \cite{J4}, \cite{J5}, \cite{J6}, \cite{J7}, \cite{J8}. $p$-adic multiple zeta values are defined as the $p$-adic analogues of the integrals (\ref{eq:2}). The papers cited above start with the question of finding $p$-adic analogues of the sums of series (\ref{eq:1}) which satisfy certain precise requirements. 

In \cite{J4}, we study the consequence of the double shuffle relation of multiple polylogarithms on the coefficients of the power series expansion at $0$. In particular we define the notion of multiple harmonic value, the numbers 
$$ \har_{\mathcal{P}^{\mathbb{N}}}(n_{1},\ldots,n_{d}) =
\bigg( {(p^{\alpha})}^{n_{1}+\ldots+n_{d}} \sum_{0<m_{1}<\ldots<m_{d}<p^{\alpha}} \frac{1}{m_{1}^{n_{1}} \ldots m_{d}^{n_{d}}} \bigg)_{(p,\alpha)  \in  \mathcal{P}_{N} \times \mathbb{N}} \in \big( \prod_{p\in\mathcal{P}} \mathbb{Q}_{p} \big)^{\mathbb{N}} $$
where $\mathcal{P}$ is the set of prime numbers. We show that the double shuffle of multiple polylogarithms gave rise to a varianst of the double shuffle relations for these numbers, which we called the harmonic double shuffle relations. The termnology multiple harmonic values comes from multiple harmonic sums, the followng rational numbers :

$$ \sum_{0<m_{1}<\ldots<m_{d}<m} \frac{1}{m_{1}^{n_{1}} \cdots m_{d}^{n_{d}}} $$

More generally, in the study of $p$-adic multiple zeta values, we have seen that \emph{by taking power series expansions, algebraic relations can be transferred, from multiple polylogarithms to certain sequences multiple harmonic sums}, where multiple harmonic sums appear as coefficients of that power series expansion.

\subsection{Outline}

The goal of this work is to generalize the above idea. Instead of taking only power series expansions at $0$, we take power series expansion at all points at the same time. And instead of working on $\mathbb{P}^{1} - \{0,1,\infty\}$ or 
$\mathbb{P}^{1} - \{0,\mu_{N},\infty\}$ as done in those papers, we will work with the moduli space $\mathcal{M}_{0,l}$.

When we derive several times an iterated integral of algebraic functions, we obtain a linear combination of iterated integrals over the ring of algebraic functions. 

By Chen's theorem on iterated integrals \cite{Chen}, multiple polylogarithms are linearly independent over this ring. Thus, an algebraic relation among multiple polylogarithms, derived several times, gives a relation among algebraic functions.

As a result, we will transfer algebraic properties of multiple polylogarithms, not to multiple harmonic sums, but to algebraic functions over $\mathbb{Q}$.

In particular, let us start with the generic double shuffle relations (\ref{eq:generic1}) (\ref{eq:generic2}). Let us apply the operators of partial derivation with respect to each variable $z_{i}$. If we consider the purely algebraic term of the relation, for instance, it gives a relation among certain algebraic functions. 

By restricting to a subvariety defined by the equation ``the generic double shuffle derived a certain number of times equals the generic double shuffle'', we get :

\textbf{Theorem.} \emph{There exist an infinite sequence of explicit closed affine subvarieties $\mathcal{Y}_{0,n+3}^{(N)}$ of $\mathcal{M}_{0,n+3}$ ($N \in \mathbb{N}^{\ast}$, $n \in \mathbb{N}$) equipped with global algebraic functions that are solutions to the generic double shuffle equations.}

This project and the above result have been announced in a previous version of \cite{J6}. More details and consequences of the result will appear in the next version of this text. We know that the double shuffle equations (\ref{eq:first shuffle}) (\ref{eq:second shuffle}) can be obtained as limits of the generic double shuffle equations (\ref{eq:generic1}) (\ref{eq:generic2}). We will see in a subsequent work what can be obtained as limits of the specification of the generic double shuffle equations constructed in this paper. We will also study the periods arising from the above subvarieties.

\emph{Acknowledgments.} This work has been done at Universit\'{e} de Strasbourg, with support of Labex IRMIA, and  Universit\'{e} de Genève, with support of NCCR SwissMAP.

\section{The de Rham pro-unipotent fundamental groupoid of $\mathcal{M}_{0,n}$ and its localization}

\subsection{The de Rham pro-unipotent fundamental groupoid of $\mathcal{M}_{0,n}$}

The notion of de Rham pro-unipotent fundamental groupoid is introduced in \cite{Deligne}.
We describe the example of 
$$ \mathcal{M}_{0,L+3} = \{ (x_{1},\ldots,x_{L+3}) \in (\mathbb{P}^{1})^{L+3} \text{ } |\text{ } x_{i}\not= x_{j} \text{ for all i} \not= \text{j} \} / \PGL_{2} $$
which is isomorphic to 
$$ \mathcal{M}_{0,L+3} = \{(y_{1},y_{2},\ldots,y_{L}) \in (\mathbb{P}^{1} \setminus \{0,1,\infty\})^{L} \text{ }|\text{ for all i,j}, y_{i} \not= y_{j} \} $$

and can be written as $\overline{\mathcal{M}_{0,L+3}} - \partial \overline{\mathcal{M}_{0,L+3}}$ where $\overline{\mathcal{M}_{0,L+3}}$ is its Deligne-Mumford compactification and $\partial \overline{\mathcal{M}_{0,L+3}}$ is a normal crossings divisor.

The de Rham pro-unipotent fundamental groupoid of $X=\mathcal{M}_{0,L+3}$ is a groupoid in affine schemes over $\mathcal{M}_{0,L}$, whose base-points are points of $\mathcal{M}_{0,L}$ plus tangential base-points in the sense of \cite{Deligne}, \S15. By \S12.4 of \cite{Deligne} one also has a canonical base-point $\omega_{\dR}$. 
For any two base-points $x,y$, one has an affine scheme 
$\pi_{1}^{\un,\dR}(X,y,x)$ and for any three base-points $x,y,z$, one has a morphism of schemes, the groupoid multiplication $\pi_{1}^{\un,\dR}(X,z,y) \times \pi_{1}^{\un,\dR}(X,y,x) \rightarrow \pi_{1}^{\un,\dR}(X,z,x)$. This makes each 
$\pi_{1}^{\un,\dR}(X,x)= \pi_{1}^{\un,\dR}(X,x,x)$ into a group scheme and each $\pi_{1}^{\un,\dR}(X,y,x)$ into a bitorsor under $(\pi_{1}^{\un,\dR}(X,x),\pi_{1}^{\un,\dR}(X,y))$.

For any two base-points $x,y$, one has an isomorphism 
$\pi_{1}^{\un,\dR}(X,y,x) \simeq \pi_{1}^{\un,\dR}(X,\omega_{\dR})$ and these isomorphisms are compatible to the groupoid structure. Thus, describing $\pi_{1}^{\un,\dR}(X)$ is reduced to describing $\pi_{1}^{\un,\dR}(X,\omega_{\dR})$.

$\pi_{1}^{\un,\dR}(X,\omega_{\dR})$ is the exponential of the pro-nilpotent Lie algebra with generators $e_{ij}$, $1 \leq i\not= j \leq L+3$ and relations 
$e_{ij} = e_{ji}$, $[e_{ij},e_{ik}+e_{jk}] = 0$ for $i,j,k$ distinct, and $[e_{ij},e_{kl}]=0$ for $i,j,k,l$ distinct.

The bundle $\pi_{1}^{\un,\dR}(X,\omega_{\dR}) \times X$ carries the Knizhnik-Zamolodchikov connection, given by 
\begin{equation} \label{eq:nablaKZ M0,n} \nabla_{\KZ} : f \mapsto df - \sum_{1 \leq i<j \leq n+3} e_{ij} d\log(x_{i}-x_{j}) f ,
\end{equation}
and, in the cubic coordinates $c_{i}$ defined by $y_{i} = c_{i} \ldots c_{L}$, $(1 \leq i \leq L)$,
\begin{equation} \nabla_{\KZ} : f \mapsto df
- \bigg( \sum_{u=1}^{L} \frac{dc_{u}}{c_{u}} \sum_{u\leq i<j \leq L} e_{ij}
- \sum_{\substack{1\leq v \leq v' \leq L \\  2 \leq i \leq j }} \frac{d(c_{v} \ldots c_{v'})}{c_{v} \ldots c_{v'} -1} e_{v-1,v'}  
-  \sum_{\substack{1\leq v \leq v' \leq L \\ 1 \leq i \leq j}} \frac{d(c_{v} \ldots c_{v'})}{c_{v} \ldots c_{v'} -1} e_{v',n-1} \bigg) f.
\end{equation}

\subsection{The localization of de Rham pro-unipotent fundamental groupoid of $\mathcal{M}_{0,n}$}

Iterated integrals such as (\ref{eq:2}) or (\ref{eq:2Li}) are obtained by applying several times operators of the form 
$$ f \mapsto \int f\omega $$
i.e. multiplying by a differential form and integrating. We will use the term localization to takl about a settng in which we consider also the inverse of these operators at the same time. This setting the functions obtained by applying several times partial derivatives of multiple polylogarithms. We have met this object, implicitly and explcitly, in other papers : \cite{J2} with the localized multiple harmonic sums of \S4, and also \cite{J8} with the adjoint $p$-adic multiple zeta values at tuples of not-all-positive integers.

The localization of a not necessarily commutative ring $R$ with respect to a multiplicative subset $S$ is defined as the ring which represents the subfunctor of $\Hom(R,-)$ defined by the homomorphisms which map $S$ to units. The representability of this functor follows from its continuity and from the fact that it satisfies the solution set condition. Explicitly, it consists of sums of elements of the form $r_{1}s_{1}^{-1}r_{2}s_{2}^{-1}\ldots$ where $r_{i} \in R$ and $s_{i} \in S$.

If $z_{0}$, $z_{l+1}$, and $z_{l+2}$  denote respectively the marked point fixed at $0$, $1$, and $\infty$, then we associate in particular, for all $i \in \{1,\ldots,L\}$ :
$e_{i,0} \leftrightarrow \big( f \mapsto \int f \frac{dz_{i}}{z_{i}} \big)$, $e_{i,n+1} \leftrightarrow \big( f \mapsto \int f \frac{dz_{i}}{z_{i}-1} \big)$ 
\noindent where integration is with respect to the variable $z_{i}$. In particular, we have formal inverses of these operators ,
$e^{inv}_{i,0} \leftrightarrow \big( f \mapsto z_{i}\frac{\partial f}{\partial z_{i}} \big)$, $ e^{inv}_{i,n+1} \leftrightarrow \big(f \mapsto(z_{i}-1)  \frac{\partial f}{\partial z_{i}} \big)$, and thus by considering the difference between the two we obtain $\frac{\partial}{\partial z_{i}}= e^{inv}_{i,0} - e^{inv}_{i,n+1}$.

\section{Multiple polylogarithms and localized multiple polylogarithms}

Multiple polylogarithms, on the one hand englobe in some sense all homotopy-invariant iterated integrals over $\mathcal{M}_{0,n}$, and on the other hand are given by simple formulas.

\subsection{Review of Goncharov's multiple polylogarithms}

Here we review some of the basic analytic properties of multiple polylogarithms of \cite{Goncharov}. Hyperlogarithms (\cite{Goncharov}, \S2.1) are the following functions on $\mathcal{M}_{0,n}$ : 
$$ \I(a_{0};a_{1},\ldots,a_{m};a_{m+1}) = \int_{a_{0}\leq a_{1} \leq \ldots \leq a_{m} \leq a_{m+1}} \frac{dt_{1}}{t_{1}-a_{1}} \wedge \frac{dt_{2}}{t_{2}-a_{2}} \wedge \ldots \wedge \frac{dt_{m}}{t_{m}-a_{m}} $$

\noindent By affine change of variable one can reduce their study to the case where $a_{0} = 0$ and $a_{n+1} = 1$ and view this indeed as a function of $(a_{1},\ldots,a_{m})$. In particular one has (\cite{Goncharov}, \S2.1)
\begin{center} $\I(z_{0};z_{1},\ldots,z_{m};z_{m+1})= \I(0;z_{1}-z_{0},\ldots,z_{m}-z_{0};z_{m+1}-z_{0})$ \end{center}

One denotes the sequence $(z_{1},\ldots,z_{m})$ by distinguishing the $z_{i}$'s that are equal to $0$, assuming that $z_{1}\not=0$, as $(z_{1},\ldots,z_{m}) = (z_{i_{1}},\underbrace{0,\ldots,0}_{n_{1}-1},\ldots z_{i_{d}},\underbrace{0,\ldots,0}_{n_{d}-1})$.

The Knizhnik-Zamolodchikov differential equation is satisfied by hyperlogarithms, in the following sense (\cite{Goncharov}, \S2.2, Theorem 2.1) :
$$ dI(z_{0};z_{1},\ldots,z_{m};z_{m+1}) = \sum_{i=1}^{m} I(z_{0};z_{1},\ldots,\hat{z_{i}},\ldots,z_{m};z_{m+1}) \big( d\log(z_{i+1}-z_{i}) - d\log(z_{i-1} - z_{i}) \big) $$

There is a simple formula expressing these iterated integrals and multiple polylogarithms in terms of each other (\cite{Goncharov}, Theorem 2.2)

If $|z_{1}-z_{0}|$ is small enough we have
$$ \I(z_{0};a_{1},\ldots,a_{n},z_{1}) = \sum_{0<m_{1}<\ldots < m_{d}} 
\frac{\big(\frac{a_{i_{2}}-z_{0}}{a_{i_{1}}-z_{0}}\big)^{n_{1}} \ldots 
	\big(\frac{a_{i_{d}}-z_{0}}{a_{i_{d-1}}-z_{0}}\big)^{n_{d-1}}
	\big(\frac{z_{1}-z_{0}}{a_{i_{d}}-z_{0}}\big)^{n_{d}}}{m_{1}^{n_{1}} \ldots m_{d}^{n_{d}}} $$

\subsection{Localized multiple polylogarithms}

We generalize the notation
$\I(z_{0};z_{1},\ldots,\widehat{z_{i}},\ldots,z_{n};z_{n+1})$ :

\begin{Notation} Let $w= \{u_{1},\ldots,u_{r}\} \subset \{1,\ldots,n\}$ (with $u_{1}< \ldots < u_{r}$), and let  $w^{c} = \{v_{1},\ldots,v_{s}\} \subset \{1,\ldots,n\}$ the complement of $w$ in $\{1,\ldots,n\}$ (with $v_{1}< \ldots < v_{s}$). We adopt the four following abreviations for the iterated integral $\I(z_{0};\ldots,\widehat{z_{v_{1}}},\ldots,\widehat{z_{v_{s}}},\ldots;z_{n+1})$ :
$$\I(w) = \I(\widehat{w^{c}}) = \I(z_{u_{1}},\ldots,z_{u_{r}}) = \I(\widehat{z_{v_{1}}},\ldots,\widehat{z_{v_{s}}}) $$
\end{Notation}

\noindent Now let us take a permutation $\imath : \{1,\ldots,n\} \hookrightarrow \{1,\ldots,n\}$, as well as a map $\frak{n} : \{1,\ldots,n\} \rightarrow \mathbb{N}$. We denote by $i_{j} = i(j)$ for all $j$, and by $n_{j}=\frak{n}(j)$. We want to write a formula for $\partial^{n_{i_{k}}}_{z_{i_{k}}} \ldots \partial^{n_{i_{1}}}_{z_{i_{1}}} \I$, by induction on $k$ and on $\frak{n}$ with the lexicographical order on $\mathbb{N}^{n}$.

\begin{Proposition} We have 
\begin{equation} \label{eq: appearance of F} 
\partial^{d_{i_{k}}}_{z_{i_{k}}} \ldots \partial^{d_{i_{1}}}_{z_{i_{1}}} \I = \sum_{w \subset \{1,\ldots,n\}}
F^{(n_{i_{k}},\ldots,n_{i_{1}})}_{w,(z_{i_{k}},\ldots,z_{i_{1}})} \I(w) 
\end{equation}
\noindent where $F^{(n_{i_{k}},\ldots,n_{i_{1}})}_{w,(z_{i_{k}},\ldots,z_{i_{1}})} \in \mathbb{Q}(z_{1},\ldots,z_{n})$ is given inductively by, for $w = \{j_{1},\ldots,j_{r}\}$ with $j_{1}<\ldots<j_{r}$, and with $j_{0} = 0$, $j_{r+1}=n+1$ :
\begin{equation} \label{eq:rational KZ system} F_{w,(i_{k},\ldots,i_{1})}^{(n_{k}+1,\ldots,n_{1})} = \partial_{z_{i_{k}}}F_{w,(i_{k},\ldots,i_{1})}^{(n_{k},\ldots,n_{1})} + 
\sum_{u=0}^{r} \sum_{i=j_{u}+1}^{j_{u+1}-1} F_{w \cup \{i\},(i_{k},\ldots,i_{1})}^{(n_{k},\ldots,n_{1})} \big( \frac{1}{z_{j_{u}} - z_{i}} - \frac{1}{z_{j_{u+1}} - z_{i}} \big) 
\end{equation}
\end{Proposition}

\begin{proof} By induction from the fact that the KZ differential equation for hyperlogarithms amounts to the following formulas : for $2 \leq j \leq n-1$,  
	$$ \frac{\partial \I}{\partial z_{j}}= \frac{1}{z_{j-1}-z_{j}} \big( \I(\widehat{z_{j}}) - \I(\widehat{z_{j-1}}) \big) - 
	\frac{1}{z_{j+1}-z_{j}} \big( \I(\widehat{z_{j}}) - \I(\widehat{z_{j+1}}) \big)   $$
	$$ \frac{\partial \I}{\partial z_{1}} = \frac{1}{-z_{1}} \big( \I(\widehat{z_{1}}) \big) - 
	\frac{1}{z_{2}-z_{1}} \big( \I(\widehat{z_{1}}) - \I(\widehat{z_{2}}) \big)  $$
	$$ \frac{\partial \I}{\partial z_{n}} = \frac{1}{z_{n-1}-z_{n}} \big( \I(\widehat{z_{n}}) - \I(\widehat{z_{n-1}}) \big) - 
	\frac{1}{1-z_{n}} \big( \I(\widehat{z_{n}}) \big) $$
\end{proof}
	
This gives all of the maps $F$, because $ F_{w,(i_{k},\ldots,i_{1})}^{(0,n_{k-1}\ldots,n_{1})} =  F_{w,(i_{k-1},\ldots,i_{1})}^{(n_{k-1},\ldots,n_{1})}$.

\begin{Remark} It is possible to turn the recursive formula into a non-recursive one. This requires to introduce some combinatorial objects such as certain trees that serve for the indexation of the formulas.
\end{Remark}

\begin{Example} In weight $n=1$ :
\begin{center} $\partial_{a} \I(a) = - \big( \frac{1}{a - z_{0}} - \frac{1}{Z_{1} - z_{2}} \big) = - \frac{(z_{2}-z_{0})}{(z_{2}-a)(a-z_{0})}$ \end{center}
\begin{center} $\partial^{n_{1}}_{a} \I(a) = - \big[ \frac{(-1)^{n_{1}-1} (n_{1}-1)!}{(a - z_{0})^{n}} - \frac{(-1)^{n_{1}-1} (n_{1}-1)!}{(a - z_{2})^{n}} \big] =
(-1)^{n_{1}}(n_{1}-1)! \big[ \frac{(a-z_{2})^{n} - (a-z_{0})^{n} }{(Z_{1}-z_{0})^{n}(Z_{1}-z_{2})^{n}} \big]$ \end{center}
\end{Example}
\begin{Example} In weight $2$ :
	\newline (i) first order
	\begin{center} $\partial_{a}\I(a,b) =\big( \frac{1}{a-b} - \frac{1}{a-z_{0}} \big) \I(b)- \big( \frac{1}{a-b}\big) \I(a)$ 
	\end{center}
	\begin{center} $\partial_{b} \I(a,b) = \big( \frac{1}{b-z_{1}} - \frac{1}{b-a} \big) \I(a) + \frac{1}{b-a} \I(b)$
	\end{center}
	\noindent (ii) second order :
	\begin{center}
		$\partial^{2}_{a}\I(a,b) = \big( \frac{1}{(a-z_{0})^{2}} - \frac{1}{(a-b)^{2}} \big) \I(b) + \frac{1}{(a-b)^{2}}\I(a) - \frac{(z_{1}-z_{0})}{(b-a)(z_{1}-a)(a-z_{0})}$ \end{center}
	\begin{center}
		$\partial^{2}_{b}\I(a,b) = \big( \frac{1}{(a-b)^{2}} - \frac{1}{(a-z_{1})^{2}} \big) \I(a) - \frac{1}{(b-a)^{2}}\I(b) - \frac{(z_{1}-z_{0})}{(b-a)(z_{1}-b)(b-z_{0})}$ \end{center}
	\begin{center}
		$\partial_{a}\partial_{b} \I(a,b) =\frac{1}{(b-a)^{2}} \big( \I(a) - \I(b) \big) - \frac{z_{1}-z_{0}}{(z_{1}-b)(b-a)(a-z_{0})}$ \end{center}
\end{Example}

\begin{Example} In weight $3$ : 
	\newline (i) first order 
	\begin{center} $\partial_{a} \I(a,b,c) = \big(\frac{1}{a-b} - \frac{1}{a - z_{0}} \big) \I(b,c) - \frac{1}{a-b} \I(a,c)$ 
	\end{center}
	\begin{center} $\partial_{c} \I(a,b,c) = \big( \frac{1}{c-z_{1}} - \frac{1}{c-b}\big) \I(a,b) + \frac{1}{c-b}\I(a,c)$ \end{center}
	\begin{center} $\partial_{b} \I(a,b,c)= \big( \frac{1}{a-b} - \frac{1}{c-b} \big) \I(a,c) + \frac{1}{b-a} \I(b,c) + \frac{1}{c-b} \I(a,b)$ \end{center}
	\noindent (ii) second order 
\begin{center} $\partial_{a}\partial_{b} \I(a,b,c) =
\frac{1}{(a-b)^{2}}\big( \I(b,c) - \I(a,c)\big) 
+ \frac{1}{(c-b)(b-a)(a-z_{0})} \big( (c-z_{0})\I(c) -(b-z_{0})\I(b) \big)$ 
\end{center}
\begin{center} $\partial_{b}\partial_{c} \I(a,b,c) = \frac{1}{(c-b)^{2}}\big( \I(a,c)- \I(a,b) \big) + \frac{1}{(b-a)(c-b)(c-z_{1})} \big( (z_{1}-b)\I(b) - (z_{1}-a) \I(a) \big)$ \end{center}
\begin{center} $\partial_{a}\partial_{c} \I(a,b,c) = -\frac{(z_{1}-a)}{(z_{1}-c)(c-a)(b-a)}\I(a) + \frac{(z_{1}-b)(b-z_{0})}{(z_{1}-c)(c-b)(b-a)(a-z_{0})}\I(b) + \frac{(c-z_{0})}{(c-b)(c-a)(z_{0}-a)}\I(c)$ \end{center}
	\noindent (iii) third order
\begin{center} $\partial_{a}\partial_{b}\partial_{c} \I(a,b,c) =
		- \frac{(z_{1}-z_{0})}{(z_{1}-c)(c-b)(b-a)(a-z_{0})} 
		- \big( \frac{(z_{1}-a)}{(c-z_{1})(c-a)(a-b)^{2}} \big)\I(a)
		+ \big( \frac{(z_{1}-b)}{(a-b)^{2}(c-z_{1})(c-b)} - \frac{(b-z_{0})}{(c-b)^{2}(b-a)(a-z_{0})} \big) \I(b)
		+ \big( \frac{z_{0}-c}{(c-b)^{2}(b-a)(a-z_{0})}\big) \I(c)$ \end{center}
\end{Example}

\section{Algebraic relations among localized multiple polylogarithms}

\subsection{Review of the generic double shuffle relations for multiple polylogarithms}

Let $A= (a_{i})_{i \in I}$ be any alphabet. The shuffle product on the $\mathbb{Q}$-monoid of words $\mathbb{Q}\langle A\rangle$ over $A$ is the bilinear map $\mathbb{Q}\langle A \rangle \times \mathbb{Q}\langle A \rangle \rightarrow \mathbb{Q}\langle A \rangle$ defined by
	\begin{center} $(a_{i_{1}}\ldots a_{i_{l}})\text{ }\sh\text{ }(a_{i_{l+1}} \ldots a_{i_{l+l'}}) =
		\sum_{\substack{\sigma \text{ permutation of }\{1,\ldots,l+l'\} \\\text{s.t. } \sigma(1)<\ldots<\sigma(l)\\ \text{and } \sigma(l+l)<\ldots<\sigma(l+l')}} 
		a_{i_{\sigma^{-1}(1)}} \ldots a_{i_{\sigma^{-1}(l+l')}}$ \end{center}

\noindent It follows from the definition of multiple polylogarithms that we have the shuffle relation :
	$$ I(a_{0};a_{1},\ldots,a_{n};a_{n+1})I(a_{0};b_{1},\ldots,b_{m};a_{n+1}) = 
	I(a_{0}; a_{1}\ldots a_{n}\text{ }\sh\text{ }b_{1}\ldots b_{m} ; a_{n+1}) $$

Let $Z$ be a finite subset of $\mathbb{P}^{1}(\mathbb{C})$ containing $\{0,1,\infty\}$. Let $Y_{Z}$ be the alphabet $\{ y_{n,x}\text{ }|\text{ }n \in \mathbb{N}^{\ast}, x \in Z- \{infty\}\}$. We identify words on $Y_{Z}$ to the words $w$ on $e_{Z}$ such that $\tilde{\partial}_{e_{0}}(w)= 0$ via the correspondence $y_{n_{1},x_{1}} \ldots y_{n_{d},x_{d}} \leftrightarrow e_{0}^{n_{d}-1}e_{x_{d}} \ldots e_{0}^{n_{1}-1}e_{x_{1}}$.

Let $\mathcal{O}^{\ast,Z}$ be the $\mathbb{Q}$-vector space freely generated by words over $Y_{Z}$, including the empty word $1$.

If $Z,Z'$ are two finite subsets of $\mathbb{P}^{1}(\mathbb{C})$ containing $\{0,1,\infty\}$, what we will denote by $ZZ'$ is simply their product set, i.e. $ZZ' = \{zz',\text{ } z \in \mathbb{Z}, z' \in \mathbb{Z'}\}$.

The shuffle product is a map 
$\mathcal{O}^{\ast,Z} \times \mathcal{O}^{\ast,Z'} \rightarrow \mathcal{O}^{\ast,ZZ'}$. For convenience we extend slightly $\mathcal{O}^{\ast,Z}$ and $\mathcal{O}^{\ast,Z'}$. Let $\mathcal{O}^{\ast,Z}_{ext}$ be the vector space freely generated by the empty sequence and the sequences $( (n_{i})_{d};(x_{i})_{d+1} ) = ( x_{d+1},n_{d},\ldots,x_{2},n_{1},x_{1})$ with $n_{i}$ positive integers and $x_{i}$ in $Z - \{0,\infty\}$. Let the injective linear map $i : \mathcal{O}^{\ast,Z} \hookrightarrow \mathcal{O}^{\ast,Z}_{ext}$ defined by $y_{n_{d},x_{d}} \ldots y_{n_{1},x_{1}} \mapsto (1,n_{d},x_{d}\ldots,n_{1},x_{1})$ and let 
$r : \mathcal{O}^{\ast,Z}_{ext} \rightarrow \mathcal{O}^{\ast,Z}$ be defined by
$(x_{d+1},n_{d},x_{d}\ldots,n_{1},x_{1}) \mapsto (1,n_{d},x_{d}/x_{d+1},\ldots,n_{1},x_{1}/x_{d+1})$.

The generalized shuffle product $\ast : \mathcal{O}^{\ast,Z}_{ext} \times \mathcal{O}^{\ast,Z'}_{ext}  \rightarrow \mathcal{O}^{\ast,ZZ'}_{ext}$ is the unique bilinear map given on couples of words as follows. 

A quasi-shuffle element of $\big((n_{1},\ldots,n_{d}),(n'_{1},\ldots,n'_{d'})\big)$
is a term of $\big((n_{1},\ldots,n_{d}) \ast (n'_{1},\ldots,n'_{d'})\big)$ where $\ast$ is as in the usual sense on $\mathbb{P}^{1} - \{0,1,\infty\}$.

For two sequences :
$w = (x_{i_{d+1}},n_{d},\ldots,x_{i_{2}},n_{1},x_{i_{1}})$,
$w' = (x_{j_{d'+1}},t_{d'},\ldots,x_{j_{2}},t_{1},x_{j_{1}})$, the quasi-shuffle product $w \ast w'$ is the sum of all the quasi-shuffle elements of (w,w'), where, for each quasi-shuffle element $(u_{d''},\ldots,u_{1})$ of 
$\big((n_{1},\ldots,n_{d}),(n'_{1},\ldots,n'_{d'})\big)$, one has e quasi-shuffle element  $(z_{i_{a_{d''+1}}}z_{j_{b_{d''+1}}},u_{d''},\ldots,z_{i_{a_{2}}}z_{j_{b_{2}}},u_{1},z_{i_{a_{1}}}z_{j_{b_{1}}})$ of $(w,w')$ defined by induction as follows :
\newline (i) $a_{1} = b_{1} = 1$
\newline (ii) for $i \in \{1,\ldots,d''-1\}$ : $\left\{ \begin{array}{l}
\text{ if }u_{i}=n_{l},\text{ } 1 \leq l \leq d,\text{ then }
 (a_{i+1},b_{i+1}) = (a_{i}+1 ,b_{i})
\\ \text{ if }u_{i}=n'_{l'},1 \leq l' \leq d',\text{ then }
 (a_{i+1},b_{i+1}) = (a_{i}, b_{i}+1)
\\ \text{ if }u_{i}=n_{l}+n'_{l'},1 \leq l \leq d,1 \leq l' \leq d',\text{ then }
 (a_{i+1},b_{i+1}) = (a_{i}+1,b_{i}+1) 
\end{array}\right.$.

It follows from the power series expansion of multiple polylogarithms that we have, for all indices $w,w'$,

$$ \Li(w)\Li(w') = \Li(w \ast w') $$

\subsection{Application to localized multiple polylogarithms}

We now see the consequences of generic double shuffle relations on localized multiple polylogarithms. We will add in the next versions more explicit versions of these equations as well as similar results about the homographical transformations. We note that some related results, on iterated integrals on $\mathbb{P}^{1} - \{0,1,z,\infty\}$ (without the iteration of the derivation) can be found in \cite{Hirose Sato}.

For simplicitly, we restrict to the case in which we apply to iterated integrals only operators of partial derivation with respect to a variable, and to which each operator on each variable $z_{i}$ is applied the same number of times. Of course, a more general construction can be done.

\begin{Notation} Let $\partial^{d}_{all\text{ }z_{i}}$ 
be $\frac{\partial}{\partial z_{1}} \ldots \frac{\partial}{\partial z_{L}}$ applied $d$ times, to a linear combination of iterated integrals which involves the variables $z_{i}$. (We do not include the variables $z_{0}$ and $z_{m+1}$).
\end{Notation}

By derivating several times the shuffle relation we obtain the following :

\begin{Proposition} We have, for all $d \in \mathbb{N}^{\ast}$ :
$$ \partial^{d}_{all\text{ }a_{i}}\I(a_{0};a_{1},\ldots,a_{n};a_{n+1})\text{ }\text{ }
\partial^{d}_{all\text{ }b_{i}}\I(a_{0};b_{1},\ldots,b_{m};a_{n+1}) = $$
$$ \partial^{d}_{all\text{ }a_{i},b_{i}}\I(a_{0};(a_{1},\ldots,a_{n}) \sh (b_{1},\ldots,b_{m});a_{n+1}) $$ 
\end{Proposition}

\begin{Example} weight $(1,1)$. The relation
\begin{center} $\partial_{a}\I(a)\partial_{b}\I(b) =  \partial_{a}\partial_{b}\I(a,b) + \partial_{a}\partial_{b}\I(b,a)  $ \end{center}
\noindent amounts to, by the cancellation on the non-purely rational terms :
\begin{center} $\frac{(z_{1}-z_{0})}{(z_{1}-a)(a-z_{0})} \frac{(z_{1}-z_{0})}{(z_{1}-b)(b-z_{0})} = - \frac{(z_{1}-z_{0})}{(z_{1}-b)(b-a)(a-z_{0})} - \frac{(z_{1}-z_{0})}{(z_{1}-a)(a-b)(b-z_{0})}$ \end{center}
which can be proven directly.
\end{Example}

\begin{Example} weight $(1,2)$. A similar description can be given for the relation :
\begin{center} $\partial_{a}\I(a)\partial_{b,c}\I(b,c) =  \partial_{a}\partial_{b}\I(a,b) + \partial_{a}\partial_{b}\I(b,a)  $ \end{center}
\end{Example}

By derivating several times the quasi-shuffle relation we obtain a similar equation.

\begin{Proposition} We have, for all $d \in \mathbb{N}^{\ast}$ :
	$$ \partial^{d}_{all}\I(a_{0};a_{1},\ldots,a_{n};a_{n+1})\text{ }\text{ }
	\partial^{d}_{all}\I(a_{0};b_{1},\ldots,b_{m};a_{n+1}) = $$
	$$ \partial^{all}\I(a_{0};(a_{1},\ldots,a_{n}) \sh (b_{1},\ldots,b_{m});a_{n+1}) $$ 
\end{Proposition}

\begin{Example} (in weight $1 \times 1$) The quasi-shuffle relation is 
\begin{center} 
$\sum_{0<m_{1}} \frac{(z/a)^{m_{1}}}{m_{1}^{n_{1}}}\sum_{0<m_{2}}\frac{(z/b)^{m_{2}}}{m_{2}^{n_{2}}}
= \sum_{0<m} \frac{(z^{2}/ab)^{m}}{m^{n_{1}+n_{2}}}
+ \sum_{0<m_{1}<m_{2}}\frac{(z/a)^{m_{1}}(z/b)^{m_{2}}}{m_{1}^{n_{1}}m_{2}^{n_{2}}}
+ \sum_{0<m_{1}<m_{2}}\frac{(z/b)^{m_{1}}(z/a)^{m_{2}}}{m_{1}^{n_{1}}m_{2}^{n_{2}}}$ 
\end{center}
\begin{center}
$= \sum_{0<m} \frac{(z^{2}/ab)^{n}}{m^{n_{1}+n_{2}}}
+ \sum_{0<m_{1}<m_{2}}\frac{(zb/ab)^{m_{1}}(z^{2}/bz)^{m_{2}}}{m_{1}^{n_{1}}m_{2}^{n_{2}}}
+ \sum_{0<m_{1}<m_{2}}\frac{(za/ba)^{m_{1}}(z^{2}/za)^{m_{2}}}{m_{1}^{n_{1}}m_{2}^{n_{2}}}$
\end{center}
\noindent in a summary :
\begin{center} 
$ \Li( \begin{array}{c} 0,a,z \\ n_{1} \end{array}) \Li( \begin{array}{c} 0,b,z \\ n_{2} \end{array} )= 
\Li( \begin{array}{c} 0,ab,z^{2} \\ n_{1}+n_{2} \end{array}) + 
\Li( \begin{array}{c} 0,ab,zb,z^{2} \\ n_{1},n_{2} \end{array}) +
\Li( \begin{array}{c} 0,ab,za,z^{2} \\ n_{2},n_{1} \end{array})$
\end{center}
\noindent We take $n_{1}=n_{2}=1$.
\begin{center}
$\I( 0,a,z ) \I( 0,b,z ) = \I( 0,ab,0,z^{2}) + \I( 0,ab,zb,z^{2}) +\I( 0,ab,za,z^{2})$
\end{center}
\noindent We apply $\partial_{a}\partial_{b}$. The left-hand side, for example, maps by $\partial_{a}\partial_{b}$ to $\big( \frac{1}{z-a} - \frac{1}{a} \big)\big( \frac{1}{z-b} - \frac{1}{b} \big)$.
\end{Example}

\section{Transfer of algebraic relations among localized multiple polylogarithms to algebraic functions}

\subsection{Definition}

Localized multiple polylogarithms arise as linear combinations of iterated integrals whose coefficients are rationals fractions over $\mathbb{Q}$.

\begin{Definition} We call $I_{N}^{\text{rat}}(a_{0};a_{1},\ldots,a_{n};a_{n})$ the purely algebraic term of that linear combination for $\partial^{N}_{all}I(a_{0};a_{1},\ldots,a_{n};a_{n+1})$, i.e. the coefficient of the constant iterated integral 1 (the iterated integral of the empty sequence of differential forms) in the linear combination.
\end{Definition}

This is a rational fraction over $\mathbb{Q}$. Of course we can also consider the other terms.

\begin{Proposition} $I_{N}^{\text{rat}}(a_{0};a_{1},\ldots,a_{n};a_{n})$ satisfies a variant of the generic double shuffle relations.
\end{Proposition}

\begin{proof} We consider the variant of the generic double shuffle equations for localized multiple polylogarithms described in the previous paragraph, and Chen's theorem \cite{Chen} which implies that multiple polylogarithms are linearly independent over algebraic functions on $\mathcal{M}_{0,L+3}$. 
\end{proof}

\begin{Definition} Let $\mathcal{Y}_{L+3}^{(N)}$ be the subvariety of the affine variety $\mathcal{M}_{0,L+3}$ defined by the equation saying that the variant of the generic double shuffle relation satisfied by  $I_{N}^{\text{rat}}(a_{0};a_{1},\ldots,a_{m};a_{m+1})$ is equal to the generic double shuffle relation.
\end{Definition}

By the previous proposition, on this subvariety, one has explicit rational fractions over $\mathbb{Q}$ which satisfy the generic double shuffle relations.

\begin{Remark} We have restricted to the operators $\partial^{N}_{z_{m}} \ldots \partial^{N}_{z_{1}}$, but we can do a similar construction with the operators $\partial^{N_{m}}_{z_{m}} \ldots \partial^{N_{1}}_{z_{1}}$, with $N_{1},\ldots,N_{n}$ not necessarily equal to each other.
\end{Remark}

In our study of $p$-adic multiple zeta values, we obtain a variant of the motivic Galois theory of $p$-adic multiple zeta values, which contains certain sums of series (multiple harmonic values) are studied as if they were periods in a certain sense. Here we propose to go towards a variant of the motivic Galois theory of multiple polylogarithms which gives a special role to certain algebraic functions.

\end{document}